%% file: main.tex
\documentclass[11pt]{amsart}

\usepackage{amsmath, amsthm, amssymb}
\usepackage{lmodern}
\usepackage{palatino}                       
\usepackage[bigdelims,vvarbb]{newpxmath}    
\usepackage[scaled=0.95]{inconsolata}       
\usepackage[T1]{fontenc}                    

\usepackage{comment}

\usepackage[abs]{overpic}		  
\usepackage{import}
\usepackage{transparent}

\usepackage{xcolor}
\definecolor{lred}{RGB}{226, 106, 106}
\definecolor{nred}{RGB}{237, 28, 36}
\definecolor{lblue}{RGB}{52, 152, 219}
\definecolor{nblue}{RGB}{0, 174, 239}
\definecolor{lyellow}{RGB}{232, 197, 91}
\definecolor{dgreen}{RGB}{0, 148, 68}
\definecolor{l1yellow}{RGB}{217, 224, 33}
\definecolor{lgrey}{RGB}{179, 179, 179}
\definecolor{indigo}{rgb}{0.29, 0.0, 0.51}  
\usepackage[colorlinks, urlcolor=indigo, linkcolor=indigo, citecolor=indigo]{hyperref}

\usepackage[hcentering, vcentering, total={5.9in, 8.2in}]{geometry}  

\usepackage{tikz-cd}
\usetikzlibrary{cd}
\tikzset{
  symbol/.style={
    draw=none,
    every to/.append style={
      edge node={node [sloped, allow upside down, auto=false]{$#1$}}
    },
  },
}
\usetikzlibrary{trees}

\theoremstyle{plain}

\newtheorem{theorem}{Theorem}
\newtheorem{corollary}[theorem]{Corollary}

\newtheorem{lemma}[theorem]{Lemma}

\newtheorem{ques}[theorem]{Question}

\theoremstyle{definition}
\newtheorem{definition}[theorem]{Definition}

\theoremstyle{remark}
\newtheorem{remark}[theorem]{Remark}

\numberwithin{theorem}{section}

\newcommand{\Q}{\mathbb{Q}}           
\newcommand{\Z}{\mathbb{Z}}           
\newcommand{\PD}{\mathrm{PD}\,} 

\makeatletter
\newcommand*\bigcdot{\mathpalette\bigcdot@{0.6}}
\newcommand*\bigcdot@[2]{\mathbin{\vcenter{\hbox{\scalebox{#2}{$\m@th#1\bullet$}}}}}
\makeatother

\newcommand{\Pushoff}{{\def\svgwidth{1,6ex}\,\,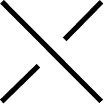\,\,}} 

\DeclareMathOperator\tb{tb}                               

\DeclareFontFamily{U} {cmr}{}
\DeclareFontShape{U}{cmr}{m}{n}{
  <-6> cmr5
  <6-7> cmr6
  <7-8> cmr7
  <8-9> cmr8
  <9-10> cmr9
  <10-12> cmr8
  <12-> cmr9}{}
\DeclareSymbolFont{Xcmr} {U} {cmr}{m}{n}

\title{Contact Surgery Numbers of the $3$-Torus}

\author{Prerak Deep}
\address{Indian Statistical Institute, Kolkata, India}
\email{prerakdg@gmail.com}

\author{Monika Yadav}
\address{Institute of Mathematics, University of Warsaw, Poland}
\email{monika.yadav9413@gmail.com}

\date{\today}

\begin{document}

\begin{abstract}
    We study contact surgery numbers for contact structures on the $3$-torus. We show that all contact structures obtained via contact surgery along a Legendrian Borromean ring are overtwisted, and that this construction yields infinitely many pairwise non-contactomorphic contact structures on $\mathbb{T}^3$. We prove an obstruction on the maximal Thurston–Bennequin invariant of $3$-component links to produce $\mathbb{T}^3$ by Dehn surgery. We also compare Legendrian surgery presentations of the Stein-fillable contact structure on $\mathbb{T}^3$. In this setting, we exhibit a non-destabilizable Legendrian knot with no graded normal ruling whose associated Stein filling has nonzero full symplectic homology.
\end{abstract}

\maketitle

\section{Introduction}
Contact surgeries along Legendrian links are a fundamental method to construct contact $3$-manifolds. One of the celebrated theorems is by Ding and Geiges that states that every contact $3$-manifold with a coorientable contact structure admits a contact surgery diagram. However, little is understood about how complicated the contact surgery diagrams are for a given contact manifold. To get a better understanding, the study of contact surgery numbers was initiated in \cite{KegelEtnyreOnaran}. The \emph{contact surgery number} $cs(M,\xi)$ of a contact $3$-manifold $(M,\xi)$ is defined as the minimal number of components in a Legendrian link in $(\mathbb{S}^3,\xi_{std})$ such that rational contact surgery (with non-vanishing contact surgery coefficients) along it yields $(M,\xi)$. Analogously, one defines $cs_{\mathbb{Z}}(M,\xi)$ and $cs_{\pm1}(M,\xi)$ by restricting the conact surgery coefficient to be integers or $\pm1$, respectively. Later, the contact surgery numbers were studied for some families of lens spaces $l(4m+3,4)$, the Brieskorn sphere $\Sigma(2,3,11)$ in \cite{ChatterjeeKegel}, and for Projective spaces in \cite{KegelYadav}. 

The purpose of this paper is to investigate contact surgery numbers for the $3$-torus $\mathbb{T}^3$. The manifold $\mathbb{T}^3$ is particularly interesting from the perspective of contact topology. It admits infinitely many pairwise non-contactomorhic ( in fact non-isotopic) tight contact structures as well as infinitely many overtwisted contact structures. Since $b_1(\mathbb{T}^3))= 3$, every surgery presentation of $\mathbb{T}^3$ requires at least three components. This lower bound is attained in
the smooth category by zero-framed surgery on the Borromean rings. The contact setting
raises two related questions: which contact structures admit a three-component surgery
presentation, and which can be obtained from Legendrian realizations of this particular
smooth link?  In \cite{KegelEtnyreOnaran}, it was proved that $cs_{\pm 1}(\mathbb{T}^3,\xi)\leq 4$ for any tight contact structure $\xi$ on $\mathbb{T}^3$. Beyond this upper bound, however, little is known about minimal surgery presentations or about the contact structures obtained from Legendrian realizations of the Borromean rings. Although their zero-framed surgery gives a
minimal smooth presentation of $\mathbb{T}^3$, we prove that every contact structure obtained from
a Legendrian realization of this presentation is overtwisted. Nevertheless, these diagrams realize every even Euler class and produce infinitely many pairwise non-contactomorphic
contact structures. Each of these structures has integer contact surgery number three.

\begin{theorem}\label{lem3}
    Let $\xi$ be any contact structure on $\mathbb{T}^3$, obtained by contact surgery along a Legendrian Borromean ring. Then
    \begin{enumerate}
        \item $\xi$ is overtwisted.
         \item Whenever $PD(\xi)=0$, then $d_3(\xi)=\frac{3}{2}$.
        \item $PD(\xi)=2(i,j,k)$, where $i,j,k\in \mathbb{Z}$. Moreover, for every $3$-tuple $(i,j,k)\in \mathbb{Z}$, we get a contact structure by performing contact surgery along the Legendrian Borromean ring with Euler class $2(i,j,k)$.
    \end{enumerate}
\end{theorem}
A direct corollary of the Theorem \ref{lem3} is the following: 
\begin{corollary}
    There is a unique contact structure $(\mathbb{T}^3,\xi)$ obtained by surgery along the Borromean ring with $PD(e(\xi))=0$. 
\end{corollary}

\begin{corollary}\label{cor3}
    There exists aleast one overtwisted contact structure $\xi$ with $PD(e(\xi))=0$ and  $cs_{\mathbb{Z}}(\mathbb{T}^3,\xi)=4$. There exist infinitely many overtwisted contact structures $\xi$ with $PD(e(\xi))=0$ and  $cs_{\mathbb{Z}}(\mathbb{T}^3,\xi)\geq 4$.
\end{corollary}
 
 Besides determining the minimal number of contact surgery components, it is natural to ask to what extent contact surgery diagrams are unique. For the unique Stein-fillable contact structure $(\mathbb{T}^3,\xi_1)$, a family of Legendrian surgery diagrams was constructed in \cite{Ding_Geiges_Diffeotopy} and \cite{KegelEtnyreOnaran}. We construct another contact surgery diagram and show that the corresponding Legendrian knot is smoothly isotopic, but not Legendrian isotopic, to one of the previously known examples. This establishes the non-uniqueness of Legendrian surgery diagrams for the tight contact structure $(\mathbb{T}^3,\xi_1)$.
 The distinction between these knots also provides an
application to graded normal rulings and symplectic
homology. Leverson's criterion \cite{Leverson} implies that a graded
normal ruling of a Legendrian attaching link guarantees
nonvanishing of the full symplectic homology of the
associated Weinstein domain. Earlier examples of
nonvanishing symplectic homology without graded augmentations appear in
Capovilla-Searle's work \cite{Capovilla}. Our construction exhibits this
phenomenon in a Legendrian surgery presentation of
$(\mathbb{T}^3,\xi_1)$.

\begin{theorem}\label{stein}
    Let $\title{L}$ be the Legendrian knot in $(\#^2\mathbb{S}^1\times\mathbb{S}^2,\xi_{std})$ as shown in Figure \ref{fig:kirby}. The total symplectic homology of the Stein manifold obtained by $(-1)$-surgery along $\widetilde{L}$ is non-zero. Moreover, the total symplectic homology of any Stein filling of $(\mathbb{T}^3,\xi_1)$ is non-zero.
\end{theorem}

The comparison knot admits no graded normal ruling, but
its contact $(-1)$-surgery also gives a Stein filling of
$(\mathbb{T}^3,\xi_1)$. The nonvanishing statement for all
Stein fillings therefore applies to this presentation as
well. Consequently, the absence of a graded normal ruling
for the attaching Legendrian knot does not force the total symplectic
homology of the resulting Stein filling to vanish.

\begin{corollary}\label{cor5}
    There exists a non-destabilizable Legendrian knot $L$ in $(\#^2\mathbb{S}^1\times\mathbb{S}^2,\xi_{std})$ such that $L$ admits no graded rulings and the Stein manifold obtained by contact $(-1)$-surgery along it has total symplectic homology non-zero.
\end{corollary}

To understand the limitations of three-component surgery presentations beyond the Borromean
rings, we study the restrictions imposed by surgery to  $\mathbb{T}^3$. Since every such diagram must satisfy some topological constraints, one expects corresponding restrictions on the classical invariants of Legendrian representatives. Our first obstruction is noteworthy, it relates
this smooth surgery condition to the maximal Thurston–Bennequin invariants of the link
and its components.

\begin{theorem}\label{tb_obstruction}
    Let $L=K^1\sqcup K^2\sqcup K^3$ be a link with $ 3$ components. If surgery along $L$ produces $\mathbb{T}^3$ and $\max{tb(L)}=\sum_{i}\max{tb(K^i)}$, then the maximal $tb(K^i)<1$ for atleast one $i$.
\end{theorem}

\begin{remark}
Note that, in general, for a Legendrian link
$L=\bigsqcup_{i=1}^n L_i,$
the Thurston--Bennequin invariant satisfies
\[
tb(L)=\sum_{i=1}^n tb(L_i)+2\sum_{1\leq i<j\leq n} lk(L_i,L_j).
\]
However, throughout this paper we mostly consider Legendrian links with three components whose pairwise linking numbers are all zero. Therefore, in these cases,
\[
tb(L)=\sum_{i=1}^n tb(L_i).
\]
\end{remark}

Combining surgery calculus with the computation of the Euler class and the $d_3$-invariant, we obtain restrictions on contact structures on $\mathbb{T}^3$ that admit surgery diagrams with only three components.

\begin{theorem}\label{thm_cs3}
    
    Let $\xi$ be a contact structure on $\mathbb{T}^3$ with integer contact surgery number $3$. If $PD(e(\xi))=0$ then $d_3(\xi)\in \{\pm \frac{1}{2},\frac{3}{2}\}$.
\end{theorem}

As consequences of the above theorem, we obtain an obstruction on the classical invariants of Legendrian surgery links producing tight contact structures on $\mathbb{T}^3$, together with finiteness results for overtwisted structures having contact surgery number $3$. 

\begin{corollary}\label{cor1}
   
    Let $K^1\sqcup K^2\sqcup K^3$ be a Legendrian link such that contact surgery along it produces $(\mathbb{T}^3,\xi)$. If $\xi$ is tight, then (possibly after re-enumeration), $tb(K^1)\geq 1,tb(K^i)\leq - 1, rot(K^i)=\pm (tb(K^i)-1),\: i=2,3$.
\end{corollary}

\begin{corollary}\label{cor4}
   
   Let $L=K^1\sqcup K^2\sqcup K^3$ be a Legendrian link such that contact $(\pm 1)$-surgery along it produces $(\mathbb{T}^3,\xi)$. If $\xi$ is tight, then (possibly after re-enumeration), $tb(K^1)=1,tb(K^2),tb(K^3)= -1,\: rot(K^i)=0 \: \forall i $. Moreover, $L$ cannot be obtained from another Legendrian link by stabilizing the $K^2$ and $K^3$-components.

\end{corollary}

\begin{corollary}\label{cor2}
There are atmost $3$ overtiwsted contact structures on $\mathbb{T}^3$ with $cs_{\pm 1}(\mathbb{T}^3,\xi)=3$ and $PD(e(\xi))=0$.
\end{corollary}
     


We conclude this section with the following questions.

\begin{ques}
    Find out the contact surgery diagrams for tight contact structures  $(\mathbb{T}^3,\xi_n)$, $n>1$. Is the $cs_{\mathbb{Z}}(\mathbb{T}^3,\xi_n)=3$ for all $n$?
\end{ques}

\begin{ques}
   Let $L=K^1\sqcup K^2\sqcup K^3$ be a link with $3$-components with $\Bar{tb}(L)=\sum_{i=1}^3\Bar{tb}(K^i)$. If surgery along $L$ produces $\mathbb{T}^3$, is it true that the maximal $\Bar{tb(K^i)}<1$ for at least two $i$ s?
\end{ques}

\section{Acknowledgement}
MY is supported by the National Science Centre (NCN), Poland, under the OPUS grant 2024/53/B/ST1/03470 at the University of Warsaw. PD wants to thank the Statistics and Mathematics Unit, Indian Statistical Institute, Kolkata, for their support. The authors thank Marc Kegel for his careful reading of an earlier version of the manuscript and for many helpful comments that improved the exposition.

\section{Preliminaries}

We briefly review the necessary background on contact Dehn surgery, a more detailed discussion can be found in \cite{Gompf_Stein,DingGeiges,Ding_Geiges_Stipsciz_surgery_diagrams,Ozbagci_Stipsicz_book,Geiges,Kegel-Durst,Kegel_thesis,Kegel_knot_complement,KegelEtnyreOnaran,Casals_Etnyre_Kegel}.

Let $K$ be a Legendrian knot in a contact $3$-manifold $(M,\xi)$. There exists a tubular neighbourhood $\nu K$ of $K$ with a convex boundary. Let $\mu$ denote the meridian of $K$ and $\lambda_c$ be the \emph{contact longitude} of $K$, which is obtained by pushing $K$ to the boundary of $\nu K$ in the direction of the Reeb vector field. A \emph{contact Dehn surgery}  with contact surgery coefficient $p/q$ along $K$ is performed by removing a standard tubular neighbourhood of $K$ and glueing back a solid torus $\mathbb{S}^1\times D^2$ via a diffeomorphism sending $\{pt\}\times \partial D^2$ to $p\mu+q\lambda_c$.

The resulting manifold admits contact structures that agree with $\xi$ outside the surgery region and are tight on the attached solid torus. The resulting contact structure need not be unique for a general contact surgery coefficient $p/q$; rather, there are finitely many possibilities of contact structures determined by the continued fraction expansion of $p/q$ (see \cite{DingGeiges}). However, Honda~\cite{Honda} and Giroux~\cite{Giroux_lens} showed that such contact structures on $\mathbb{S}^1\times \mathbb{D}^2$ always exist and are unique when $p=\pm1$.

Note that the \textit{Seifert longitude} $\lambda_s$, obtained by pushing $K$ into a Seifert surface, satisfies the relation $\lambda_c=\lambda_s+\tb(K)\mu$
where $\tb(K)$ is the Thurston–Bennequin invariant of $K$. Thus the \textit{topological surgery coefficient} $r_t$, which is measured with respect to the Seifert longitude $\lambda_s$, and the contact surgery coefficient $r_c$ are related by $r_c=r_t-\tb(K)$. Note that all contact surgery diagrams representing contact $(\pm1)$-surgery
along Legendrian knots in $(\mathbb{S}^3,\xi_{\mathrm{std}})$ are drawn as
front ($xz$) projections in $(\mathbb{R}^3,\xi_{\mathrm{std}}
=\ker(dz-y\,dx))$.

Now we state the following theorem, already mentioned in the introduction.

\begin{theorem}[Ding–Geiges \cite{DingGeiges}]\label{thm:lik-wallace-contact}
Let $(M,\xi)$ be a contact 3-manifold. Then $(M,\xi)$ can be obtained by rational contact Dehn surgery along a Legendrian link in $(\mathbb{S}^3,\xi_{std})$. Moreover, one can assume all contact surgery coefficients to be of the form $(\pm1)$.\qed
\end{theorem}

Before we state the lemmas needed to convert contact surgery coefficients from integers to rational numbers and vice versa, we introduce some notation. For $m,n\in \mathbb{N}0$, let $K_n$ denote a Legendrian knot which is obtained by adding $n$ stabilizations to $K$ and further $K_{n,m}$ denote a Legendrian knot which is obtained by adding $m$ extra stabilizations to $K_n$. Let $K{\def\svgwidth{1,6ex}\,\,\input{PushOff.pdf_tex}\,\,} K$ denote the Legendrian link consisting of $K$ and a Legendrian knot obtained by pushing $K$ in the direction transverse to contact planes.

\begin{lemma} [Ding–Geiges~\cite{Ding_Geiges_torus_bundles,Ding_Geiges_Stipsciz_surgery_diagrams}]\label{lem:Kirby}
Let $K$ be a Legendrian knot in $(\mathbb{S}^3,\xi_{std})$.
\begin{enumerate}
\item \textbf{Cancellation lemma:} For all $n\in \mathbb{Z}-\{0\}$, we have
$$K\left(\frac{1}{n}\right){\def\svgwidth{1,6ex}\,\,\input{PushOff.pdf_tex}\,\,} K\left(-\frac{1}{n}\right)\cong (\mathbb{S}^3,\xi_{std}).$$
\item \textbf{Replacement lemma:} For all $n\in \mathbb{Z}-\{0\}$, we have
$$K\left(\pm\frac{1}{n}\right)\cong K(\pm 1){\def\svgwidth{1,6ex}\,\,\input{PushOff.pdf_tex}\,\,}\cdots{\def\svgwidth{1,6ex}\,\,\input{PushOff.pdf_tex}\,\,} K(\pm 1).$$
\item \textbf{Translation lemma:} For $r\in \mathbb{Q}-\{0\}$ and $k \in \mathbb{Z}$, we have
$$K(r)\cong K\left(\frac{1}{k}\right){\def\svgwidth{1,6ex}\,\,\input{PushOff.pdf_tex}\,\,} K\left(\frac{1}{
\frac{1}{r}-k}\right).$$
In the case when $r<0$, $r$ can be uniquely written as
\begin{align*}
r=[r_1+1,\ldots,r_n]:=r_1+1-\frac{1}{r_2-\frac{1}{\cdots-\frac{1}{r_n}}}
\end{align*} with integers $r_1,\ldots,r_n\leq -2$ and then we have
\begin{align*}
K(r)\cong K_{|2+r_1|}(-1){\def\svgwidth{1,6ex}\,\,\input{PushOff.pdf_tex}\,\,} K_{|2+r_1|,|2+r_2|}(-1){\def\svgwidth{1,6ex}\,\,\input{PushOff.pdf_tex}\,\,}
\cdots {\def\svgwidth{1,6ex}\,\,\input{PushOff.pdf_tex}\,\,} K_{|2+r_1|,|2+r_2|,\ldots, |2+r_n| }(-1).
\end{align*}
\end{enumerate}
In addition, all these results hold in a tubular neighbourhood of $K$. In particular, they can be applied to knots in larger contact surgery diagrams.\qed
\end{lemma}

Next, we continue our discussion with the homotopical invariants of the underlying tangential $2$-plane field of a contact structure. It is known that a tangential $2$-plane field $\xi$ on $M$ is completely determined (up to homotopy) by the $d_3$-invariant and Gompf’s $\Gamma$-invariant~\cite{Gompf_Stein,DingGeigesStipsciz}.
The $\Gamma$-invariant encodes $\xi$ on the $2$-skeleton of $M$, while the $d_3$-invariant specifies $\xi$ on the $3$-cell. The $\Gamma$-invariant is equivalent to the Euler class whenever there is no $2$-torsion in $H_1(M;\mathbb{Z})$. Since in our case, $H_1(\mathbb{T}^3,\mathbb{Z})=\mathbb{Z}\oplus \mathbb{Z}\oplus \mathbb{Z}$, we will be only working with the Poincaré dual of the Euler class $e$ and the $d_3$-invariant.

The contact structure obtained by contact surgery with coefficients $\pm \frac{1}{n}, n\in \mathbb{N}$ is unique. Therefore, whenever we have a contact surgery diagram with contact surgery coefficients $\pm \frac{1}{n}$, we can compute the homotopical invariants $e$ and $d_3$-invariant using the following lemma ~\cite{Ding_Geiges_Stipsciz_surgery_diagrams,Kegel-Durst}.

\begin{lemma}\label{lem:d3}
Let $L=L_1\cup\cdots\cup L_k$ be an oriented Legendrian link in $(S^3,\xi_{std})$ and let $(M,\xi)$ be the contact manifold obtained by contact $\left(\pm\frac{1}{n}\right)$-surgeries along $L$. Let $t_i$ and $r_i$ be the Thurston–Bennequin invariant and rotation number of $L_i$ for all $i=1,\ldots,k$. Let $l_{ij}$ be the linking number of $L_i$ with $L_j$ and let $\frac{p_i}{q_i}=t_i\pm \frac{1}{n_i}$ be the topological surgery coefficient of $L_i$. We define the linking matrix $Q$ by
\begin{align*}
Q=\begin{pmatrix}
p_1&q_2l_{12}&\cdots& q_kl_{1k}\\
q_1l_{12}&p_2&\cdots &q_kl_{2k}\\
\vdots&\vdots&\ddots&\vdots\\
q_1l_{1k}& q_2l_{2k}&\cdots &p_k
\end{pmatrix} .
\end{align*}
\begin{enumerate}
\item The first homology $H_1(M;\Z)$ is presented  by the abelian group $\langle\mu_i| Q\mu^T=0\rangle$, where $\mu=(\mu_1,\ldots,\mu_k)$ is the vector of meridians $\mu_i$ of $L_i$.
\item The Poincaré dual of the Euler class is given by $PD(e(\xi))=\sum_{i=1}^k n_ir_i\mu_i$.
\item If there exists a rational solution $b \in \Q^k$ of $Qb = r$, where $r=(r_1,\ldots,r_k)^T$, the $d_3$-invariant is well defined and is computed as
$$\frac{1}{4}{\langle r,b\rangle-3\sigma(Q)-2\operatorname{rk}(Q)-2}-\frac{1}{2}+q$$
where $\langle \cdot,\cdot\rangle$ simply denotes the vector dot product, $\sigma(Q)$ and $\operatorname{rk}$ denote the signature and rank of $Q$, and $q$ is the number of positive contact surgeries in the diagram. Equivalently, the formula for the $d_3$-invariant in terms of the surgery coefficients is given by
\[\frac{1}{4}\sum_{i=1}^k \left(n_ib_ir_i+(3-n_i)sign(n_i)\right)-\frac{3\sigma(Q)}{4}.\]
\qed
\end{enumerate}
\end{lemma}

Next, we discuss some background on Legendrian knots in $\#^k(\mathbb{S}^1\times\mathbb{S}^2)$ and their ruling invariants, which will be needed later for Theorem \ref{stein}. There is a unique tight contact structure on $\#^k(\mathbb{S}^1\times\mathbb{S}^2),\: k\geq 0$. In this paper, we consider Legendrian links in $\#^k(\mathbb{S}^1\times\mathbb{S}^2)$ with the standard contact structure. Every Legendrian link in $\#^k(\mathbb{S}^1\times\mathbb{S}^2)$ can be described in the $xy$- and $xz$-diagrams using tangles in normal form as defined by Gompf \cite{Gompf_Stein}. An example is shown in Figure \ref{fig:kirby}. For this paper, we only work with $xz$-diagrams.

\begin{definition}
Let $A,M>0$, consider $[0,A]\times[-M,M]\times[-M,M]$ with the contact structure $dz-ydx$. A Legendrian tangle $T$ in $A\times[-M,M]\times[-M,M]$ is said to be in \emph{normal form} if $T$ meets $x=0$ and $x=A$ in $k$ groups of strands, where the groups are of size $N_1,\ldots, N_l$ in the $xz$- projections, and within each of the $l$-th group, the strands are labeled by $1,\ldots,N_l$ from top to bottom at $x=0$ and $x=A$.
\end{definition}

For every Legendrian link in $\#^k(\mathbb{S}^1\times\mathbb{S}^2)$ there exists a tangle in the normal form such that joining the strands with the same label in each group gives the given Legendrian link (see \cite{Gompf_Stein}). Also, we can pass from an $xz$-diagram in normal form to a contact surgery diagram by joining the strands with the same label and replacing each $1$-handle with the standard Legendrian unknot; see Figure~\ref{fig:kirby}.

In \cite{Ekholm_Ng},  Ng and Ekholm, defined a DGA generated over $\mathbb{Z}[t_1^{\pm1},\ldots,t^{\pm1}_s]$ assocaited with a Legendrian link ($s$-components) in $\#^k(\mathbb{S}^1\times\mathbb{S}^2)$. They showed that the homology associated with the DGA is an invariant of Legendrian isotopy. Later, in \cite{Leverson}, Leverson defined $\rho$-graded normal rulings for Legendrian links and $\rho$-graded augmentation associated with their DGAs and showed that the existence of a $\rho$-graded ruling and the number of $\rho$-graded normal rulings is an invariant of Legendrian knots. Before defining the normal rulings, first we describe the Maslov potential associated with the $xz$-diagram in the normal form below.

Let $L=\sqcup_{i=1}^s L_i$ be a Legendrian link in $\#^k(\mathbb{S}^1\times\mathbb{S}^2)$ and let $r(L):=gcd(rot(L_1),\ldots, rot(L_s))$. Let $T$ be a tangle in the $xz$-diagram of $L$ in the normal form. Then a \emph{Maslov potential} $m$ associates an integer modulo $2r(L)$ to each strand in the tangle $T$ minus cusps such that the following conditions hold:
\begin{enumerate}
    \item Each strand with the same label passing through the same $1$-handle at $x=0,A$ must have the same Maslov potential,
    \item if a strand is oriented such that it enters the $1$-handle at $x=A$ and exits at $x=0$, then
the Maslov potential of the strand must be even. Otherwise, the Maslov potential of the strand must be odd,
\item at a cusp, the upper strand has Maslov potential one more than the lower strand.
\end{enumerate}
Note that the Maslov potential is well-defined up to an overall shift by an even integer for knots. Now we define a normal ruling from \cite{Leverson}.

\begin{definition}
    Consider a Legendrian link $L$ in $\#^k(\mathbb{S}^1\times\mathbb{S}^2)$ and tangle portion of the $xz$-diagram in normal form of $L$. We consider paths in the $xz$-diagram minus the cusps. A \emph{normal ruling} is a decomposition of the ($xz$-diagram)-minus the cusps into pairs of paths with the conditions described below. Let $\alpha$ be a path and let $\alpha_c$ denote its companion path, then the pairings satisfy the following conditions.

\begin{enumerate}
\item Paths $\alpha$ and $\alpha_c$ may meet at at most one left cusp and at most one right cusp. Moreover, if they meet at both a left cusp and a right cusp, then neither $\alpha$ nor $\alpha_c$ passes through a $1$-handle.

\item If $\alpha$ passes through a $1$-handle at $x=0$, then its paired path $\alpha_c$ must pass through the same $1$-handle at $x=0$. Likewise, if $\alpha$ passes through a $1$-handle at $x=A$, then $\alpha_c$ must pass through the same $1$-handle at $x=A$.

\item If neither $\alpha$ nor $\alpha_c$ passes through a $1$-handle, then the union $\alpha\cup\alpha_c$ bounds an embedded disc in the $xz$-diagram. If $\alpha$ and $\alpha_c$ pass through a $1$-handle, then $\alpha\cup\alpha_c$, together with the corresponding $1$-handle, bounds an embedded disc in $xz$-diagram.

    \item \textbf{Crossings.}
    Paired paths do not intersect at crossings. Suppose the paths labeled $\alpha$ and $\alpha'$ meet at a crossing, then exactly one of the following occurs:
    \begin{itemize}
        \item \emph{switched crossing}: in a neighbourhood of the crossing $\alpha$ completely lies above (or below) $\alpha'$, or
        \item \emph{unswitched crossing}: At the crossing $\alpha$ and $\alpha'$ cross each other.
    \end{itemize}

    \item \textbf{Normality.}
    At every switched crossing, the companion strands must satisfy a standard normal configuration: If $\alpha$ and $\alpha'$ meet at a switched crossing, then the disc bounded by $\alpha\cup \alpha_c$ and $\alpha'\cup \alpha_c'$ in a neighbourhood of the switched crossing is either disjoint or nested.

\end{enumerate}
\end{definition}

To each crossing, in the $xz$-diagram, we can assign a \emph{grading} which is given by the Maslov potential of the upper strand minus the Maslov potential of the lower strand.

\begin{definition}
    Let $\rho| 2r(L)$. We say that a normal ruling is \emph{$\rho$-graded} if the grading of every switched crossing is divisible by $\rho$. We say a $\rho$-graded ruling is \emph{graded} if $\rho=0$.
\end{definition}

Now we state a result from \cite{Leverson}, which we will be using to prove Theorem \ref{stein}.

\begin{theorem}[Corollary 1.4, \cite{Leverson}]\label{Graded}
    If $X$ is the Weinstein $4$-manifold that results from attaching $2$-handles along a Legendrian link $L$ to $\#^k(\mathbb{S}^1\times \mathbb{S}^2)$ and $L$ has a graded normal ruling, then the full symplectic homology $S\mathbb{H}(X)$ is non-zero.
\end{theorem}

\section{Contact Surgery Diagrams for $\mathbb{T}^3$}

Consider the $3$-Torus $\mathbb{T}^3=\mathbb{R}^3/(2\pi \mathbb{Z})^3$. The classification of tight contact structures on $\mathbb{T}^3$ due to Kanda \cite{Kanda} shows that $\mathbb{T}^3$ has $\mathbb{Z}$-many infinitely many non-contactomorphic tight contact structures. Let $\xi_n$ be the contact structure given by the kernel of the $1$-form 
\[\cos(n\theta)dx-\sin (n \theta)dy,
\] then every tight contact structure on $\mathbb{T}^3$ is contactomorphic to exactly one of the $\xi_n$, and all of the $\xi_n$s lie in the same homotopy class of tangential 2-plane fields. Consequently, homotopical invariants such as the Euler class and the $d_3$-invariant do not distinguish tight contact structures on $\mathbb{T}^3$.

We begin by proving the non-uniqueness of Legendrian surgery diagrams for the unique Stein-fillable contact structure $(\mathbb{T}^3,\xi_1)$ and that the total symplectic homology of any Stein fillings of $(\mathbb{T}^3,\xi_1)$ is non-zero.

\begin{proof}[Proof of Theorem \ref{stein}]
Consider the Legendrian knot $\title{L}$ in tight $\#^2(\mathbb{S}^1\times \mathbb{S}^2)$ as shown in the bottom left of Figure \ref{fig:kirby}. Note that the Thurston-Bennequin invariant of $\widetilde{L}$ is $1$. Therefore, contact $(-1)$-surgery along $\widetilde{L}$ gives a Stein fillings of $\mathbb{T}^3$. Thus the contact surgery diagram obtained from this Kirby diagram yields the tight contact structure $\xi_1$, as it is the only Stein-fillable contact structure on $\mathbb{T}^3$ (see Figure \ref{fig:kirby}, on the right).
\begin{figure}
    \centering
    \includegraphics[width=0.5\linewidth]{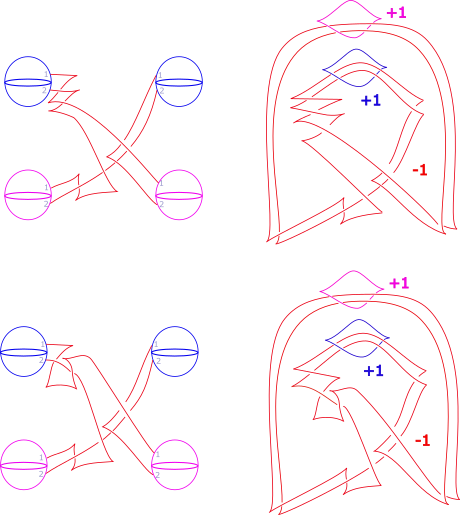}
    \caption{On the left we have two Legendrian knots $L$ (top) and $\widetilde{L}$ (bottom) in $\#^2(\mathbb{S}^1\times \mathbb{S}^2)$. On the right, we have a contact $(\pm1)$-surgery diagram for $\mathbb{T}^3$ obtained from the diagrams on the left.}
    \label{fig:kirby}
\end{figure}
\begin{figure}
    \centering
    \includegraphics[width=0.9\linewidth]{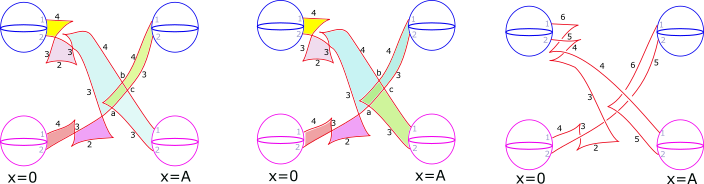}
    \caption{Two graded rulings for $\widetilde{L}$ (from the left).}
    \label{fig:ruling}
\end{figure}

Now we show that the total symplectic homology of this Stein filling is non-zero. To prove this, we show that $\widetilde{L}$ admits graded rulings. For this, first note that the rotation number of $\widetilde{L}$ is zero; thus it makes sense to compute graded rulings. In Figure \ref{fig:ruling}, the labels in black are the Maslov potential of each strand. We see that in $\widetilde{L}$ all crossings except crossing $c$ have grading $0$. And we get two graded rulings for $\widetilde{L}$: one graded ruling is obtained by making switches at all the crossings except at $b,c$, and the other graded ruling is obtained by making switches at all the crossings except at $a,c$. Using Theorem \ref{Graded}, we get that the Stein filling obtained by symplectic handle attachment along $\widetilde{L}$ has total symplectic homology non-zero.

Since we already know that there is a unique Stein filling of $(\mathbb{T}^3,\xi_1)$, any Stein filling of $(\mathbb{T}^3,\xi_1)$ has total symplectic homology non-zero.
    
\end{proof}

\begin{proof}[Proof of Corollary \ref{cor5}]
Consider the Legendrian knots $L$ in tight $\#^2(\mathbb{S}^1\times \mathbb{S}^2)$ as shown in top left of Figure \ref{fig:kirby}. Note that the rotation number of $L$ is $0$ and the Thurston-Bennequin invariant is $1$. Also, it can be seen easily by ignoring the cusp information that $L$ and $\widetilde{L}$ are smoothly isotopic. Since we already showed above that $\widetilde{L}$ admits a graded ruling, it is maximal $tb$ and hence $L$ is also maximal $tb$. Also note that since Stein fillings of $(\mathbb{T}^3,\xi_1)$ are unique up to deformation, the total symplectic homology of the Stein filling obtained from symplectic handle attachment along $L$ has total symplectic homology non-zero.

Now, we show that $L$ has no graded rulings. To see this, first we assign a Maslov potential in the $xz$-diagram of $L$ as shown in the rightmost picture in Figure \ref{fig:ruling}. We can see that the first crossing near the $1$-handle(blue) at $x=0$ does not have grading $0$. In order to have a normal ruling, it is mandatory to have a switch at this crossing. Thus, the normal rulings we get for $L$ can not be graded. 
    
\end{proof}

In the following lemma, we note the homotopical invariants of a tight contact structure on $\mathbb{T}^3$, which will be useful later.

\begin{lemma}\label{lem2}
    Let $(\mathbb{T}^3,\xi)$ be a tight contact manifold. Then $d_3(\xi)=\frac{1}{2}$ and $PD(e(\xi))=0$.
\end{lemma}
\begin{proof}
Due to the classification of tight contact structures on $\mathbb{T}^3$ \cite{Kanda}, we know that all tight contact structures on $\mathbb{T}^3$ are in the same homotopy class; therefore, it is enough to compute homotopy invariants only for one tight contact structure $\xi_1$, whose two contact $(\pm1)$-surgery diagrams are given in Figure \ref{fig:kirby}. The topmost diagram is already given in \cite{KegelEtnyreOnaran} and \cite{Ding_Geiges_Diffeotopy}.

One can see that for these contact surgery diagrams the linking matrix is zero, and the rotation numbers of link components are also zero. Therefore, $PD(e(\xi))=0$ and $d_3(\xi)=\dfrac{1}{4}\{(3-1)+(3-1)-(3-1)\}=\dfrac{1}{2}$.
\end{proof}

Now, we focus our attention on the contact structures on $\mathbb{T}^3$ that arrise my surgery along a link with $3$-components. Note that in order to get $\mathbb{T}^3$ by Dehn surgery, we need a link with atleast $3$-components since $H_1(\mathbb{T}^3;\mathbb{Z})=\mathbb{Z}\oplus\mathbb{Z}\oplus\mathbb{Z}$. One of the known examples is the Borromean ring, Dehn surgery along the Borromean ring with surgery coefficient $0$ produces $\mathbb{T}^3$, thus the contact surgery number of any contact structure on $\mathbb{T}^3$ is at least $3$. While contact surgery along the Borromean ring produces $\mathbb{T}^3$, the following lemma \ref{lem1} shows that if there is any other $3$-component link which produces $\mathbb{T}^3$ by Dehn surgery, then the link shares some common properties with the Borromean ring.
\begin{lemma}\label{lem1}
    If $\mathbb{T}^3$ is obtained as the surgered manifold $K^1(s_1)\sqcup K^2(s_2)\sqcup K^3(s_3)$ along a $3$-component link $K^1\sqcup K^2\sqcup K^3$, then $s_i=0, 
    \: \forall i$ and $lk(K^i,K^j)=0$ for all $i \neq j$.
\end{lemma}

\begin{proof}
    Let $K^1\sqcup K^2 \sqcup K^3$ be any link in $\mathbb{S}^3$ such that $K^1(s_1)\sqcup K^2(s_2)\sqcup K^3(s_3)=\mathbb{T}^3$, where $s_i=\frac{p_i}{q_i}, \: p_i,
    q_i\in \mathbb{Z}$ with $gcd(p_i,q_i)=1\: \forall \: i$ and whenever $p_i=0$ we take $ q_i=1$ and similarly whenever $q_i=0$ we take $ p_i=1$. Let $l_{ij}$ be linking number of $K^i$ and $K^j$. Then the linking matrix is 
    \[
    Q=\begin{bmatrix}
        p_1& q_2l_{12}&q_3l_{13}\\
        q_1l_{12}& p_2 &q_3l_{23}\\
        q_1l_{13}&q_2l_{23}& p_3
    \end{bmatrix},
    \]
    Let $\mu_i$ be the meridian of $K^i$, then $H_1(\mathbb{T}^3,\mathbb{Z})=\mathbb{Z}\oplus \mathbb{Z}\oplus \mathbb{Z}=\left< \mu_i| Q\mu=0\right>$, where $\mu=[\mu_1,\mu_2,\mu_3]$. As we have $Q\mu=0$, we get $p_1\mu_1+q_2l_{12}\mu_2+q_3l_{13}\mu_3=0$, thus $p_1, q_2l_{12}, q_3l_{13}=0$. Since $p_1=0$, we get $ q_1=1$. Similarly, by solving the rest of the equations obtained from $Q\mu=0$, we get  $p_2,p_3=0.$ Therfore we can assume $q_2=1=q_3$ and thus $l_{12},l_{13},l_{23}=0.$ 
\end{proof}

Next, we prove Theorem \ref{lem1}, which provides an obstruction on maximal $tb$ of link components that give $\mathbb{T}^3$ by Dehn surgery. In the proof, we mainly use the fact that contact ($-1$)-surgery along a Legendrian knot in a tight contact manifold $(M,\xi)$ produces a tight contact structure \cite{Wand}.

\begin{proof}[Proof of Theorem \ref{tb_obstruction}]
     Let $L= K^1\sqcup K^2\sqcup K^3$ be a link with with $\Bar{tb}(L)=\sum\Bar{tb}(K^i) $ and Dehn surgery along $L$ gives $\mathbb{T}^3$. Using Lemma \ref{lem1}, if $\mathbb{T}^3=K^1(s_1)\sqcup K^2(s_2)\sqcup K^3(s_3)$, then $lk(K^i,K^j)=0$ for all $i\neq j$ and $s_i=0$ for all $i$. 
     
     Now, on the contrary, assume that $\Bar{tb}(K^i)\geq 1$ for all $i$. Let $L^i$ be a Legendrian realisation of $K^i$ with $tb(L^i)\geq 1$ (Note that the condition $\Bar{tb}(L)=\sum\Bar{tb}(K^i)$ is used to gurantee existence of such an $L^i$). Let $t_i=tb(L^i)$. Then we have a contact surgery diagram 
    \begin{align*}
    L^1(-t_1)\sqcup L^2(-t_2)\sqcup L^3(-t_3)\cong L^1_{t_1-1}(-1)\sqcup L^2_{t_2-1}(-1)\sqcup L^3_{t_3-1}(-1).
\end{align*}

Then $(M,\xi'):=L^1_{t_1-1}(-1)$ is tight. Since linking of $L^2$ with $L^1$ is zero, $L^2$ remains nullhomologus in $(M,\xi')$ and $tb(L^2)$ in $(M,\xi')$ is the same as the $tb(L^2)$ in $(\mathbb{S}^3,\xi_{std})$ (see, section 4.1 and 4.2 for computing new $tb$ and homology class in the surgeres manifols, \cite{Kegel_thesis}). Therefore, contact $-1$ surgery on $L^2_{t_2-1}$ in $(M,\xi')$ is tight and it is same as $L^1_{t_1-1}(-1)\sqcup L^2_{t_2-1}(-1)$. Thus $(Y,\xi''):=L^1_{t_1-1}(-1)\sqcup L^2_{t_2-1}(-1)$ is tight. Similarly, $L^3$ remains null-homologous in $(Y,\xi'')$ as its linking number with both $L^1,L^2$ is zero, and also the Thurston--Bennequin of $L^3$ also remains the same in $(Y,\xi'')$. Now, we perform contact $(-1)$ surgery on $L^3_{t_3-1}$ in tight $(Y,\xi'')$ to get $L^1(-t_1)\sqcup L^2(-t_2)\sqcup L^3(-t_3)$. Thus, the surgered contact manifold is tight. Therefore $d_3(\xi)=\frac{1}{2}, PD(e(\xi))=0, PD(e(\xi))=0$ using Lemma \ref{lem2}. 

To arrive at a contradiction, we show that the $d_3$-invariant obtained from the contact surgery diagram is not equal to $\frac{1}{2}$. First, note that the linking matrix is zero. Thus the Euler class $0=PD(e(\xi))= r_1\mu_1+r_2\mu_2+r_3\mu_3=0$ implies $r_i=0$ for all $i$. Now using the formula for $d_3$, we get 
\begin{align*}
    d_3=\dfrac{1}{4}\{-(3-1)-(3-1)-(3-1)\}=\dfrac{-3}{2},
\end{align*}
which is a contradiction.
\end{proof}

\begin{remark}\label{rmk1}
    From above theorem, we can deduce that if $K^1\sqcup K^2 \sqcup K^3$ be a Legendrain link in $(\mathbb{S}^3,\xi_{std})$ with $tb(K^i)\geq 1\: \forall \: i$, then surgery along it can never produce $\mathbb{T}^3$.
\end{remark}

Next, we work with a general contact surgery diagram for $(\mathbb{T}^3,\xi)$ with $ 3$-components, and whenever $PD(e(\xi))=0$, we compute the $d_3$-invariant using Lemma \ref{lem:d3}. The $d_3$-invariant for a contact structure can be calculated using a contact surgery diagram only if its Euler class is torsion, that is, zero in the case of $\mathbb{T}^3$. Making use of this lemma, we compute the $d_3$-invariant of all contact structures on $\mathbb{T}^3$ whose contact surgery number is $3$.

\begin{proof}[Proof of Theorem \ref{thm_cs3}]
    Let $cs_{\mathbb{Z}}(\mathbb{T}^3,\xi)=3$. Then using Lemma \ref{lem1}, there exists a Legendrian link $K^1\sqcup K^2\sqcup K^3$ such that $(\mathbb{T}^3,\xi)=K^1(-t_1)\sqcup K^2(-t_2)\sqcup K^3(-t_3)$, where $t_i=tb(K_i)$ and $lk(K^i,K^j)=0$ for all $i\neq j$. If $t_i\leq -2$, then $K^i(-t_i)=K^i(+1)\times K^i_1\left(\frac{-1}{-t_i-1}\right)$, we always use $\mu_i, \mu_i'$ to denote the meridian of $K^i, K^i_1$, respectively and we use $r_i$ and $r_i'$ to denote the rotation numbers of $K^i$ and $K^i_1$, respectively. Note that in the definition of the contact surgery number, we consider contact surgery with a non-zero contact surgery coefficient (if one of the $t_i=0$, then the resultant contact structure is overtwisted). So we only consider cases with $t_i\neq 0 \: \forall i$. The contact surgery description depends on the values of the $ t_i$'s; therefore, we consider subcases depending on the values of the $ t_i$'s and compute the $d_3$-invariant whenever the Euler class is $0$. Also note that the case with $t_i\geq 1$ for all $i$ will not appear due to Remark \ref{rmk1}.


    \noindent \textbf{Case 1.} $t_i\leq -2, i=1,2,3$.

Using the translation Lemma, the contact surgery diagram can be written as 
\begin{align*}
    &K^1(-t_1)\sqcup K^2(-t_2)\sqcup K^3(-t_3)\cong\\
    &\left(K^1(+1)\times K^1_1\left(\frac{-1}{-t_1-1}]\right)\right)\sqcup \left(K^2(+1)\times K^2_1\left(\frac{-1}{-t_2-1}\right)\right)\sqcup \left(K^3(+1) \times K^3_1\left(\frac{-1}{-t^3-1}\right)\right),
\end{align*}
 and the linking matrix for the above contact $(\frac{\pm 1}{n})$-surgery diagram is 

 \[
    Q=\begin{bmatrix}
        t_1+1&-t_1(t_1+1)&0& 0&0&0\\
        t_1&-t_1^2&0&0&0&0\\
        0&0& t_2+1&-t_2(t_2+1)&0&0\\
        0&0&t_2&-t_2^2&0&0\\
        0&0&0&0&t_3+1&-t_3(t_3+1)\\
        0&0&0&0&t_3&-t_3^3
    \end{bmatrix}.
    \]
    The eignevalues of $Q$ are $0,0,0,-t_1^2+t_1+1,-t_2^2+t_2+1,-t_3^2+t_3+1$. Therefore, $\sigma(Q)=-3$. Let $\mu=[\mu_1,\mu_1'\mu_2,\mu_2',\mu_3,\mu_3']$, solving $Q\mu=0$, gives us $\mu_i=t_i\mu_i', i=1,2,3$. Also, note that $n_i=1, n_i'=-t_i-1$ and $sign(n_i)=+$ and $sign(n_i')=-$ and the rotation number $r_i'$ of $K_i^i$ takes value in $\{r_i+1,r_i-1\}$. for all $i$.
    
    Thus 

    \begin{align*}
        PD(e(\xi))&=\sum_{i=1}^3 r_in_i\mu_i+\sum_{i=1}^3 r_i'n'_i\mu_i'\\
        &=\sum_{i=1}^3 r_i\mu_i+\sum_{i=1}^3 r_i'(-t_i-1)\mu_i'\\
        &=\sum_{i=1}^3(r_it_i+r_i'(-t_i-1))\mu_i'.
    \end{align*}
     Note that
\begin{align*}
    PD(e(\xi))=0 &\iff r_it_i+r_i'(-t_i-1)=0 \: \forall i\\
    &\iff r_it_i+(r_i\pm 1)(-t_i-1)=0 \: \forall i\\
   &\iff  r_i =\pm(-t_i-1) \: \forall i.
\end{align*}

Now we calculate the $d_3$-invariant using Lemma \ref{lem:d3} when the Euler class is zero. Let's assume that the Euler class is $0$, which is equivalent to $r_i =\pm(-t_i-1)$. One can check that the rotation vector $r=[r_1,r_1',r_2,r_2',r_3,r_3'],$ satisfies $ Qb=r,$ where $ b=[b_1,0,b_2,0,b_3,0]$ with $b_i=\mp 1$ if $r_i'=r_i\pm 1$. Therefore we have, $r_in_ib_i=t_i+1$ and $b_i'=0$ for all $i$. Thus 

\begin{align*}
d_3&=\frac{1}{4}\left\{\sum_{i=1}^3r_ib_in_i+\sum_{i=1}^3r_i'b_i'n_i'+\sum_{i=1}^3 sign (n_i) (3-n_i)+\sum_{i=1}^3 sign(n_i')(3-n_i')\right\}-\frac{3}{4}(\sigma(Q))\\
    &=\frac{1}{4}\left\{(t_1+1)+(t_2+1)+(t_3+1)+3(3-1)-(4+t_1)-(4+t_2)-(4+t_3)\right\}-\frac{3}{4}(-3)\\
    &=\frac{1}{4}\left\{3+6-12\right\}+\frac{9}{4}=\frac{3}{2}.
\end{align*}
Thus, all the contact structures in this case are overtwisted, and the Euler class $0$ implies that $d_3=\frac{3}{2}$.

We proceed with the other cases similarly.

\noindent \textbf{Case 2.} $t_1=-1, t_2, t_3\leq -2 $.
The contact surgery diagram is given by 
\begin{align*}
    &K^1(+1)\sqcup K^2(-t_2)\sqcup K^3(-t_3)\cong\\
    &K^1(+1)\sqcup \left(K^2(+1)\times K^2_1\left(\frac{-1}{-t_2-1}\right)\right)\sqcup \left(K^3(+1) \times K^3_1\left(\frac{-1}{-t_3-1}\right)\right),
\end{align*}
 and the linking matrix is given by 

 \[
    Q=\begin{bmatrix}
        0& 0&0&0&0\\
        0& t_2+1&-t_2(t_2+1)&0&0\\
        0&t_2&-t_2^2&0&0\\
        0&0&0&t_3+1&-t_3(t_3+1)\\
        0&0&0&t_3&-t_3^3
    \end{bmatrix}.
    \]
The eignevalues of $Q$ are $0,0,0,-t_2^2+t_2+1,-t_3^2+t_3+1$, thus $\sigma(Q)=-2$. Again $PD(e(\xi))= 0$ iff $r_1=0, r_i =\pm(-t_i-1) \: i= 2,3$.

Now consider $r_1=0, r_i =\pm(-t_i-1), i=2,3$. Similar to above, one can check that the rotation vector $r=[0,r_2,r_2',r_3,r_3'],$ satisfies $ Qb=r,$ where $ b=[0,b_2,0,b_3,0]$ with $b_i=\mp 1$ if $r_i'=r_i\pm 1, i=2,3$. Therefore we have, $r_in_ib_i=t_i+1, i=2,3$ and $r_1=b_i'=0$ for all $i$. Thus 

\begin{align*}
d_3&=\frac{1}{4}\left\{(t_2+1)+(t_3+1)+3(3-1)-(4+t_2)-(4+t_3)\right\}-\frac{3}{4}(-2)\\
    &=\frac{1}{4}\left\{2+6-8\right\}+\frac{6}{4}=\frac{3}{2}.
\end{align*}

\noindent \textbf{Case 3.} $t_1,t_2=-1, t_3\leq -2$.

Using the translation lemma,

$K^1(+1)\sqcup K^2(+1)\sqcup K^3(-t_3)\cong K^1(+1)\sqcup K^2(+1)\sqcup \left(K^3(+1) \times K^3_1\left(\frac{-1}{-t_3-1}\right)\right) $. 

\[
    Q=\begin{bmatrix}
        0& 0&0&0\\
        0& 0&0&0\\
        0&0&t_3+1&-t_3(t_3+1)\\
        0&0&t_3&-t_3^3
    \end{bmatrix}.
    \]
Note that eigenvalues of $Q$ are $0,0,0,-t_3^2+t_3+1$, therfore $\sigma(Q)=-1$. We have

\begin{align*}
    PD(e(\xi))=r_3\mu_3+(r_3\pm 1)(-t_3-1)\mu_3'=(-r_3\pm (-t_3-1))\mu_3'.
\end{align*}

$PD(e(\xi))=0$ iff $r_3=\pm (-t_3-1)$. We can calculate the $d_3$-invariant. First note that $b=[0,0,\mp 1,0]$ is a solution of $Qb=r=[0,0,r_3,r_3\pm 1]$ and thus 
\begin{align*}
    d_3&=\frac{1}{4}\left\{ (t_3+1)+3(3-1)-(4+t_3)\right\}-\frac{3\sigma(Q)}{4}=\frac{3}{2}.
\end{align*}

\noindent\textbf{Case 4.} $t_1,t_2, t_3=-1$.

The contact surgery diagram is given as $K^1(+1)\sqcup K^2(+1)\sqcup K^3(+1) $ and the linking matrix $Q=0$. The Euler class $PD(e(\xi))=\sum r_i\mu_i=0$ iff $r_i=0 \forall i$. And, whenever the Euler class is $0$, one can show that $d_3=\frac{3(3-1)}{4}=\frac{3}{2}$.
    
\noindent \textbf{Case 5.} $t_1\geq 1, t_2,t_3\leq -2$.

Again, by the use of the translation Lemma, we have
    \begin{align*}
    &K^1(-t_1)\sqcup K^2(-t_2)\sqcup K^3(-t_3)\cong\\
    & K^1_{t_1-1}(-1)\sqcup \left(K^2(+1)\times K^2_1\left(\frac{-1}{-t_2-1}\right)\right)\sqcup \left(K^3(+1) \times K^3_1\left(\frac{-1}{-t_3-1}\right)\right),
\end{align*}
 and the linking matrix is the same as given in Case 2 above, with $\sigma(Q)=-2$. Solving $Q\mu=0$, gives us $\mu_i=t_i\mu_i', i=2,3$. Let $r_1'$ be the rotation number of $K^1_{t_1-1}$. Now 
    
\begin{align*}
    PD(e(\xi))=r_1'\mu_1+\sum_{i=2,3}(r_it_i+r_i'(-t_i-1))\mu_i'=0 &\iff r_1', r_it_i+r_i'(-t_i-1)=0, \: i=1,2 \\
   &\iff r_1'=0,\: r_i =\pm(-t_i-1), \:  i=2,3.
\end{align*}
     Now consider $r_1'=0,\:r_i =\pm(-t_i-1)$. As described before, $b=[0,b_2, 0, b_3, 0]$ is a solution for $Qb=r=[0, r_2, r_2', r_3, r_3']$, where $b_i=\mp 1$ if $r_i'=r_i\pm 1, i=2,3$. Thus 

\begin{align*}
    d_3&=\frac{1}{4}\left\{(t_2+1)+(t_3+1)+2(3-1)-(3-1)-(4+t_2)-(4+t_3)\right\}-\frac{3}{4}(-2)\\
    &=\frac{1}{4}\left\{2+2-8\right\}+\frac{6}{4}=\frac{1}{2}.
\end{align*}

\noindent \textbf{Case 6.} $t_1\geq 1, t_2=-1,t_3\leq -2$.
    \begin{align*}
    &K^1(-t_1)\sqcup K^2(-t_2)\sqcup K^3(-t_3)\cong\\
    & K^1_{t_1-1}(-1)\sqcup K^2(+1)\sqcup \left(K^3(+1) \times K^3_1\left(\frac{-1}{-t^3-1}\right)\right),
\end{align*}
 and the linking matrix is the same as in Case 3 with $\sigma(Q)=-1$. Solving $Q\mu=0$, gives us $\mu_3=t_3\mu_3'$.
    \begin{align*}
        PD(e(\xi))= r_1'\mu_1+r_2\mu_2+r_3\mu_3+ r_3'(-t_3-1)\mu_3'\\
        =r_1'\mu_1+r_2\mu_2+(r_3t_3+r_3'(-t_3-1))\mu_3'.
    \end{align*}
     Now $PD(e(\xi))=0$ iff $r_1',r_2=0, r_3 =\pm(-t_3-1)$, where $r_1'$ is the rotation number of $K^1_{t_1-1}$. Then $b=[0, 0, \mp1, 0]$ is a solution for $Qb=r=[0,0, r_3, r_3\pm 1]$. Thus 

\begin{align*}
    d_3&=\frac{1}{4}\left\{(t_3+1)+2(3-1)-(3-1)-(4+t_3)\right\}-\frac{3}{4}(-1)=\frac{1}{2}.
\end{align*}

\noindent \textbf{Case 7.} $t_1\geq 1, t_2,t_3=-1$.
    \begin{align*}
    K^1(-t_1)\sqcup K^2(-t_2)\sqcup K^3(-t_3)\cong K^1_{t_1-1}(-1)\sqcup K^2(+1)\sqcup K^3(+1),
\end{align*}
 and the linking matrix is given by $Q=0$ and $\sigma(Q)=0$.
    \begin{align*}
        PD(e(\xi))= r_1'\mu_1+r_2\mu_2+r_3\mu_3.
    \end{align*}
     Now $PD(e(\xi))=0$ iff $r_1',r_2,r_3=0$ and in this case, 

\begin{align*}
    d_3&=\frac{1}{4}\left\{2(3-1)-(3-1)\right\}=\frac{1}{2}.
\end{align*}

\noindent \textbf{Case 8.} $t_1, t_2\geq 1, t_3\leq -2$.

In this case the contact surgery diagram is
\begin{align*}
    &K^1(-t_1)\sqcup K^2(-t_2)\sqcup K^3(-t_3)\cong K^1_{t_1-1}(-1)\sqcup K^2_{t_2-1}(-1)\sqcup \left(K^3(+1) \times K^3_1\left(\frac{-1}{-t_3-1}\right)\right),
\end{align*}
 and the linking matrix same as the one in Case 3 with $\sigma(Q)=-1$.

 Solving $Q\mu=0$, gives us $\mu_3=t_3\mu_3'$. Let $r_1'$ and $r_2'$ be the rotation numbers of $K^1_{t_1-1}$ and $K^2_{t_2-1}$, respectively. Then
    \begin{align*}
        PD(e(\xi))= r_1'\mu_1+r_2'\mu_2+r_3\mu_3+ r_3'(-t_3-1)\mu_3'\\
        =r_1'\mu_1+r_2'\mu_2+(r_3t_3+r_3'(-t_3-1))\mu_3'.
    \end{align*}
     Now $PD(e(\xi))=0$ iff $r_1',r_2'=0, r_3 =\pm(-t_3-1)$. Then $b=[0, 0, \mp1, 0]$ is a solution for $Qb=r=[0,0, r_3, r_3\pm 1]$. Thus 

\begin{align*}
    d_3&=\frac{1}{4}\left\{(t_3+1)-2(3-1)+(3-1)-(4+t_3)\right\}-\frac{3}{4}(-1)=\frac{-1}{2}.
\end{align*}

\noindent \textbf{Case 9.} $t_1, t_2\geq 1, t_3= -1$.
\begin{align*}
    &K^1(-t_1)\sqcup K^2(-t_2)\sqcup K^3(-t_3)\cong K^1_{t_1-1}(-1)\sqcup K^2_{t_2-1}(-1)\sqcup K^3(+1),
\end{align*}
 and the linking matrix $Q=0$. The Euler class
    \begin{align*}
        PD(e(\xi))= r_1'\mu_1+r_2'\mu_2+r_3'\mu_3
    \end{align*}
    iff $r_1',r_2',r_3=0$, where $r_i'$ is rotation number of $K^i_{t_i-1}\: \forall \:i$. Then 

\begin{align*}
    d_3&=\frac{1}{4}\left\{-2(3-1)+(3-1)\right\}=\frac{-1}{2}.
\end{align*}
       
Thus in all the cases, either the Euler class is non-zero or, when it is zero, the $d_3$-invariant takes value in the set $\{\pm \frac{1}{2}, \frac{ 3}{2}\}$.
    
\end{proof}

\begin{proof}[Proof of Corollary \ref{cor1}]
    Let $(\mathbb{T}^3,\xi)$ be tight and obtained by contact surgery along a Legendrian link $K^1\sqcup K^2\sqcup K^3$. Then we know that $d_3(\xi)=\frac{1}{2}$ and $\PD(e(\xi))=0$. From Theorem \ref{thm_cs3}, we see that Euler class $0$ and $d_3=\frac{1}{2}$ only appears when $tb(K^1)\geq 1$ and $tb(K^2),tb(K^3)\leq -1$. Also, from the proof of Theorem \ref{thm_cs3}, note that $PD(e(\xi))=0\implies$ that $rot(K^i)=\pm (tb(K^i)-1)$ for $i=2,3$.

\end{proof}
\begin{proof}[Proof of Corollary \ref{cor4}]
   From the proof of Corollary \ref{cor1}, we can assume that $tb(K^1)\geq 1$ and $tb(K^2),tb(K^3)\leq -1$. Since we have contact $(\pm1)$- surgery coefficients, we get $tb(K^1)= + 1$ and $tb(K^2),tb(K^3)=-1$. The contact surgery diagram is given as 
   $K^1(-1)\sqcup K^2(+1)\sqcup K^3(+1)$.

Suppose, for contradiction, that $L$ is obtained from another Legendrian link $\widetilde{L}$ by stabilizing the $K^2$- and $K^3$-components. Using Proposition~3.2 of \cite{Lisca_Stipsicz_Contact_Szabo_invariants}, it follows that if a stabilized component in a contact $(\pm1)$-surgery diagram has surgery coefficient $+1$, then the resulting contact structure is overtwisted. Since both $K^2$ and $K^3$ have surgery coefficient $+1$, the contact structure obtained by surgery on $L$ would be overtwisted, contradicting the assumption that $\xi$ is tight. Therefore, $L$ cannot be obtained from another Legendrian link by performing stabilization on the components corresponding to $K^2$ and $K^3$.
\end{proof}

The proof of the next corollary concerning $\pm1$-contact surgery description of a tight contact structure on $\mathbb{T}^3$ follows from the fact that often $(+1)$-contact surgeries on stabilized knots result in overtwisted contact structures.

\begin{proof}[Proof od Corollary \ref{cor2}.]
    Let $cs_{\pm1}(\mathbb{T}^3,\xi)=3$. Then $(\mathbb{T}^3,\xi)=K^1(-t_1)\sqcup K^2(-t_2)\sqcup K^3(-t_3)$, where $K^i$ is a Legendrian knot with $tb(K^i)=t_i \in \{1,-1\}$. From the proof of the Theorem \ref{thm_cs3}, we can see that such case occurs in Case 4, 7, and 9 and in these cases whenever $PD(e(\xi))=0$, the invariant $d_3(\xi)=\frac{q^+-q^-}{2}\in \left\{\frac{\pm 1}{2},\frac{ 3}{2}\right\}$, where $q^{\pm}$ is the number of contact $(\pm1)$-surgeries. Also, we know that if $\xi$ is overtwisted then it is determined by its homotopy class, therefore there are atmost $3$ overtwisted contact structures $\xi$ on $\mathbb{T}^3$ with $d_3(\xi)=\frac{ 3}{2},\frac{\pm1}{2}$ and $cs_{\pm 1}(\mathbb{T}^3,\xi)=3$.
\end{proof}

Next, we investigate the contact structures obtained by surgery along Borromean rings. We show that all of them are overtwisted, we compute their Euler classes and summarize them in Table \ref{listd_3}.

\begin{proof}[Proof of Theorem \ref{lem3}]
    The proof is very similar to Case 1 of Theorem \ref{thm_cs3}, where we consider all Legendrian knots with $tb\leq -1$. Let $U^1\sqcup U^2\sqcup U^3$ be any Legendrian Borromean ring and let $(\mathbb{T}^3,\xi)$ be obtained by contact surgery along this Legendrian link. Let $t_i,r_i$ denote $tb$ and rotation number of $U^i$, respectively. Note that the maximal $tb$ of a Legendrian Borromean ring is $-4$, therefore, at least one of the knots is stabilized, \cite{Mohnke}. We fix a maximal $tb$ Legendrian Borromean ring as shown in Figure \ref{fig:max_tb_Bor}, and all Legendrian representatives of Borromean rings are obtained by stabilising this link. Note that if $t_i\leq -2$, then $U^i(-t_i)=U^i(+1)\times U^i_1\left(\frac{-1}{-t_3-1}\right)$. We use $\mu_i, \mu_i'$ to denote the meridian of $U^i, U^i_1$, respectively. Without loss of generality, assume that $t_1,t_2\leq -1, t_3\leq -2$. For $i=1,2$, whenever $t_i\leq -2$, we assume that $U^i$ is obtained from the Legendrian unknot of $tb=-1$ by doing $p_i$ positive stabilizations and $s_i$ negative stabilizations. In the case of $U^3$, we assume that it is obtained from the Legendrian unknot of $tb=-2, rot=-1$, by doing $p_3$ positive stabilizations and $s_3$ negative stabilizations (see Figure \ref{fig:max_tb_Bor}). Then $r_i=p_i-s_i, t_i=-1-p_i-s_i$ for $i=1,2$ and $r_3=-1-s_3+p_3, t_3=-2-p_3-s_3$. 

    First, we prove (1) and (2). Case 1 of Thereom \ref{thm_cs3}, we know that either $PD(e(\xi))$ non-zero or when it is zero the $d_3$-invariant takes value $\frac{3}{2}$. Therefore, in both cases $\xi$ is overtwisted.

    Now we prove (3). We make cases to study all the possibilities for $PD(e(\xi))$ depending on $t_i$.

   \begin{figure}
\centering
\includegraphics[width=0.4\linewidth]{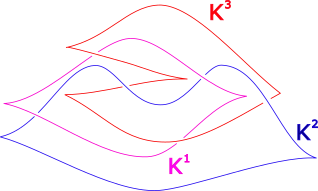}
\hspace{30pt}
\includegraphics[width=0.4\linewidth]{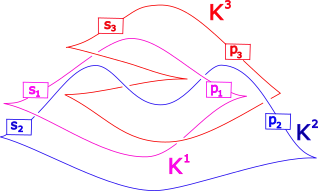}
\caption{The diagram on the left is for a maximal $tb$ Legendrian Borromean ring. The diagram on the right shows that all possible realisations of Legendrian Borromean rings are obtained by stabilizing the link components.}
\label{fig:max_tb_Bor}
\end{figure}

\noindent \textbf{Case 1}. $t_1, t_2, t_3\leq -2 $. 

This is similar to Case 1 in Theorem \ref{thm_cs3}. The contact surgery diagram is given by
\begin{align*}
&U^1(-t_1)\sqcup U^2(-t_2)\sqcup U^3(-t_3)\cong\\
&\left(U^1(+1)\times U^1_1\left(\frac{-1}{-t_1-1}\right)\right)\sqcup \left(U^2(+1)\times U^2_1\left(\frac{-1}{-t_2-1}\right)\right)\sqcup \left(U^3(+1) \times U^3_1\left(\frac{-1}{-t^3-1}\right)\right),
\end{align*}
and the linking matrix is the same as in Case 1 of Theorem \ref{thm_cs3}, with relations $\mu_i=t_i\mu'_i, i=1,2,3$.

Note that contact structures on the surgered manifold depend on the choice of stabilization in the pushoff $U^i_1$ of $U^i, i=1,2,3$, so we introduce some notation to distinguish between these cases. Let $\xi_{(1^{1_*},2^{2_*},3^{3_*})}$ denote the contact structure obtained by making a choice of $i_*\in \{+,-\}$ stabilization on the pushoff of $U^i$. Let $r_i'\in \{r_i+1,r_i-1\}$ be the rotation number of the pushoff $U^i_1$. Thus 
\begin{align*}
    PD(e(\xi))&=r_1\mu_1+(r_1'(-t_1-1))\mu_1'+r_2\mu_2+(r_2'(-t_2-1))\mu_2'+r_3\mu_3+r_3'(-t_3-1)\mu_3'\\
    &=(-r_1\pm (-t_1-1))\mu_1'+(-r_2\pm (-t_2-1))\mu_2'+(-r_3\pm (-t_3-1))\mu_3'.
\end{align*}

We calculate the Euler classes for each choice of the stabilization in $U^i_1$ as follows:
\begin{align*}
PD(e(\xi_{(1^+,2^+,3^+)}))&=(-r_1-t_1-1)\mu_1’+(-r_2-t_2-1)\mu_2’+(-r_3-t_3-1)\mu_3’\\
&=2s_1\mu_1’+2s_2\mu_2’+(2+2s_3)\mu_3’\\
PD(e(\xi_{(1^+,2^+,3^-)}))&=(-r_1-t_1-1)\mu_1’+(-r_2-t_2-1)\mu_2’+(-r_3+t_3+1)\mu_3’\\
&=2s_1\mu_1’+2s_2\mu_2’-2p_3\mu_3’\\
PD(e(\xi_{(1^+,2^-,3^+)}))&=(-r_1-t_1-1)\mu_1’+(-r_2+t_2+1)\mu_2’+(-r_3-t_3-1)\mu_3’\\
&=2s_1\mu_1’-2p_2\mu_2’+(2+2s_3)\mu_3’\\
PD(e(\xi_{(1^-,2^+,3^+)}))&=(-r_1+t_1+1)\mu_1’+(-r_2-t_2-1)\mu_2’+(-r_3-t_3-1)\mu_3’\\
&=-2p_1\mu_1’+2s_2\mu_2’+(2+2s_3)\mu_3’
\end{align*}
Similarly, we calculate the rest of the Euler classes:
\begin{align*}
PD(e(\xi_{(1^+,2^-,3^-)}))&=2s_1\mu_1’-2p_2\mu_2’-2p_3\mu_3’\\
PD(e(\xi_{(1^-,2^-,3^+)}))&=-2p_1\mu_1’-2p_2\mu_2’+(2+2s_3)\mu_3’\\
PD(e(\xi_{(1^-,2^+,3^-)}))&=-2p_1\mu_1’+2s_2\mu_2’-2p_3\mu_3’\\
PD(e(\xi_{(1^-,2^-,3^-)}))&=-2p_1\mu_1’-2p_2\mu_2’-2p_3\mu_3’
\end{align*}

We see that these Euler classes are non-vanishing for infinitely many values of $s_i,p_i$. Therefore, all the contact structures corresponding to these are overtwisted. We can also see that there are infinitely many different Euler classes; hence there are infinitely many non-contactomorphic overtwisted contact structures coming from this case. Now we consider the next case.

\noindent \textbf{Case 2}. $t_1=-1, t_2, t_3\leq -2 $.

The contact surgery diagram is given by
\begin{align*}
U^1(+1)\sqcup \left(U^2(+1)\times U^2_1\left(\frac{-1}{-t_2-1}\right)\right)\sqcup \left(U^3(+1) \times U^3_1\left(\frac{-1}{-t_3-1}\right)\right),
\end{align*}
and the linking matrix is already given in Theorem \ref{thm_cs3} with $\mu_i=t_i\mu_i’, i=2,3$.

Let $\xi_{(1,2^{\pm},3^{\pm})}$ denote the contact structure obtaine by making a choice of $(\pm)$ stabilization on $U^2_1$ and $U^3_1$. Note that in the case of the unknot, $r_1=0$. Then
\begin{align*}
PD(e(\xi_{(1,2^+,3^+)}))&=r_2\mu_2+(r_2+ 1)(-t_2-1)\mu_2’+r_3\mu_3+(r_3+ 1)(-t_3-1)\mu_3’\\
&=(-r_2-t_2-1)\mu_2’+(-r_3-t_3-1)\mu_3’\\
&=2s_2\mu_2’+(2+2s_3)\mu_3’\\
PD(e(\xi_{(1,2^+,3^-)}))&=r_2\mu_2+(r_2+ 1)(-t_2-1)\mu_2’+r_3\mu_3+(r_3- 1)(-t_3-1)\mu_3’\\
&=2s_2\mu_2’±2p_3\mu_3’\\
\end{align*}

After similar computations, we get
\begin{align*}
PD(e(\xi_{(2^-,3^-)}))&=-2p_2\mu_2’-2p_3\mu_3’\\
PD(e(\xi_{(2^-,3^+)}))&=-2p_2\mu_2’+(2+2s_3)\mu_3’
\end{align*}

\noindent \textbf{Case 3}. $t_2=-1, t_1, t_3\leq -2 $.

This case is similar to Case 2 above. The surgery description is as follows:
\begin{align*}
\left(U^1(+1)\times U^1_1\left(\frac{-1}{-t_1-1}\right)\right)\sqcup U_2(+1)\sqcup \left(U^3(+1)\sqcup \times U^3_1\left(\frac{-1}{-t^3-1}\right)\right).
\end{align*}

Let $\xi_{(1^{\pm},2,3^{\pm})}$ denote the contact structure obtaine by making a choice of $(\pm)$ stabilization on $U^1_1$ and $U^3_1$. We get
\begin{align*}
PD(e(\xi_{(1^+,2,3^+)}))&=2s_1\mu_1’+(2+2s_3)\mu_3’\\
PD(e(\xi_{(1^+,2,3^-)}))&=2s_1\mu_1’-2p_3\mu_3’\\
PD(e(\xi_{(1^-,2,3^-)}))&=-2p_1\mu_1’-2p_3\mu_3’\\
PD(e(\xi_{(1^-,2,3^+)}))&=-2p_1\mu_1’+(2+2s_3)\mu_3’
\end{align*}

\noindent \textbf{Case 4.} $t_1,t_2=-1, t_3\leq -2$.

Using the translation lemma, the contact surgery diagram is as follows:

$U^1(+1)\sqcup U^2(+1)\sqcup \left(U^3(+1) \times U^3_1\left(\frac{-1}{-t^3-1}\right)\right).$

Note that, from this contact surgery diagram, we obtained at most two contact structures depending on the sign of stabilization used in $U^3_1$. If the stabilization is positive, then we denote it by $\xi_{1,2,3^+}$ and otherwise we denote it by $\xi_{1,2,3^-}$.

Thus

\begin{align*}
PD(e(\xi))&=r_3\mu_3+(r_3\pm 1)(-t_3-1)\mu_3’\\
&=(-r_3\pm (-t_3-1))\mu_3’.
\end{align*}

Thus $PD(e(\xi_{1,2,3^+}))=2(1+s_3)\mu_3’ \neq 0$ and $PD(e(\xi_{1,2,3^-}))= -2p_3\mu_3’$.

We summarize all the Euler classes in Table \ref{listd_3}. We see that by varying the number of positive and negative stabilizations, $s_i,p_i\geq 0$, on the unknot components, we can get any Euler class of the form $2(i,j,k)$ for all $i,j,k\in \mathbb{Z}$.

\begin{table}[h]\label{listd_3}
\centering
\begin{tabular}{|c|c|c|}
\hline
\multicolumn{3}{|c|}{
$\begin{aligned}
t_i &= tb(U_i),\quad r_i = rot(U_i)\\
t_i &= -1-p_i-s_i,\quad r_i = p_i-s_i,\quad i=1,2\\
t_3 &= -2-p_3-s_3,\quad r_3 = -1+p_3-s_3, p_i, s_i\geq 0
\end{aligned}$
} \\ \hline
Possible values of $t_i$s & $\xi$ & Euler class $(\mu_1’,\mu_2’,\mu_3’)$ \\ \hline

$t_1,t_2,t_3\leq -2$ & $\xi_{(1^+,2^+,3^+)}$& $2(s_1,s_2,s_3+1)$ \\ \hline
$t_1,t_2,t_3\leq -2$&  $\xi_{(1^+,2^+,3^-)}$ & $2(s_1,s_2,-p_3)$ \\ \hline
$t_1,t_2,t_3\leq -2$&  $\xi_{(1^+,2^-,3^+)}$ & $2(s_1,-p_2,1+s_3)$ \\ \hline
$t_1,t_2,t_3\leq -2$& $\xi_{(1^-,2^+,3^+)}$ & $2(-p_1,s_2,s_3+1)$ \\ \hline
$t_1,t_2,t_3\leq -2$& $\xi_{(1^+,2^-,3^-)}$ & $2(s_1,-p_2,-p_3)$ \\ \hline
$t_1,t_2,t_3\leq -2$&$\xi_{(1^-,2^-,3^+)}$ & $2(-p_1,-p_2,s_3+1)$ \\ \hline
$t_1,t_2,t_3\leq -2$& $\xi_{(1^-,2^+,3^-)}$ & $2(-p_1,s_2,-p_3)$ \\ \hline
$t_1,t_2,t_3\leq -2$&$\xi_{(1^-,2^-,3^-)}$ & $2(-p_1,-p_2,-p_3)$ \\ \hline

$t_1=-1,t_2,t_3\leq -2$& $\xi_{(1,2^+,3^+)}$ & $2(0,s_2,s_3+1)$ \\ \hline
$t_1=-1,t_2,t_3\leq -2$& $\xi_{(1,2^+,3^-)}$ & $2(0,s_2,-p_3)$ \\ \hline
$t_1=-1,t_2,t_3\leq -2$& $\xi_{(1,2^-,3^+)}$ & $2(0,-p_2,s_3+1)$ \\ \hline
$t_1=-1,t_2,t_3\leq -2$& $\xi_{(1,2^-,3^-)}$ & $2(0,-p_2,-p_3)$ \\ \hline

$t_2=-1,t_1,t_3\leq -2$& $\xi_{(1^+,2,3^+)}$ & $2(s_1,0,s_3+1)$ \\ \hline
$t_2=-1,t_1,t_3\leq -2$& $\xi_{(1^+,2,3^-)}$ & $2(s_1,0,-p_3)$ \\ \hline
$t_2=-1,t_1,t_3\leq -2$& $\xi_{(1^-,2,3^+)}$ & $2(-p_1,0,s_3+1)$ \\ \hline
$t_2=-1,t_1,t_3\leq -2$& $\xi_{(1^-,2,3^-)}$ & $2(-p_1,0,-p_3)$ \\ \hline

$t_1=t_2=-1,t_3\leq -2$&$\xi_{(1,2,3^+)}$ & $2(0,0,s_3+1)$ \\ \hline
$t_1=t_2=-1,t_3\leq -2$& $\xi_{(1,2,3^-)}$ & $2(0,0,-p_3)$ \\ \hline

\end{tabular}
\caption{Euler classes obtained by contact surgery on Legendrian Borromean rings $U_1\sqcup U_2\sqcup U_3$.}
\end{table}

\end{proof}

\begin{proof}[Proof of Corollary \ref{cor3}]
    Consider the unique overtwisted contact structure $\xi_1'$ on $\mathbb{T}^3$ with $cs_{\mathbb{Z}}(\mathbb{T}^3,\xi_1')=3, \: PD(e(\xi_1'))=0$ and $d_3(\xi_1')=\frac{3}{2}$ as obtained in Theorem \ref{lem3}. Let $\xi_1$ be the unique overtwisted contact structure $\mathbb{S}^3$ which is obtained by $+1$ surgery along a Legendrian unknot of $tb=-2$ (\cite{KegelEtnyreOnaran}). Let $\xi_2'$ be the contact structure on $\mathbb{T}^3$ obtained as $(\mathbb{T}^3,\xi_1')\#(\mathbb{S}^3,\xi_1)$. We know $d_3(\xi_1)=1$, therefore $d_3(\xi_2')=\frac{5}{2}$ and also $cs_{\mathbb{Z}}(\mathbb{T}^3, \xi_2')\leq 4$. Also note that, $PD(e(\xi_2'))=0$ as the connected sum with $(\mathbb{S}^3,\xi_1)$ does not change the $\Gamma$-invariant and hence the Euler class in unchanges ($2\Gamma=e$). Thus using Corollary \ref{cor2}, we have $cs_{\mathbb{Z}}(\mathbb{T}^3,\xi_2')\geq 4$. Hence, $cs_{\mathbb{Z}}(\mathbb{T}^3,\xi_2')= 4$.
    
    Repeating the same process with $(\mathbb{T}^3,\xi_2')$ and $(\mathbb{S}^3,\xi_1)$, we can construct contact structures $\xi_3'$ on $\mathbb{T}^3$ with Euler class $0$ and $d_3(\xi_3')=\frac{7}{2}$. Again by using Corollary \ref{cor2}, we see that $cs_{\mathbb{Z}}(\mathbb{T}^3,\xi_3')\geq 4$. Inductively, we repeat the same process with $(\mathbb{T}^3,\xi_{n-1}')$ and $(\mathbb{S}^3,\xi_1)$ to obtain $(\mathbb{T}^3,\xi_{n}')$ with $PD(e(\xi_n'))=0$ and $d_3(\xi_n')=\frac{2n+1}{2}\: \forall n\geq 2$. All these contact structures are (non-contactomorphic) overtwisted, with $cs_{\mathbb{Z}}(\mathbb{T}^3,\xi_n')\geq 4\: \forall \: n\geq 4$.
\end{proof}

\let\MRhref\undefined
  \bibliographystyle{plain}
	\bibliography{bibliography.bib}

\end{document}

%% file: PushOff.pdf_tex
\begingroup%
  \makeatletter%
  \providecommand\color[2][]{%
    \errmessage{(Inkscape) Color is used for the text in Inkscape, but the package 'color.sty' is not loaded}%
    \renewcommand\color[2][]{}%
  }%
  \providecommand\transparent[1]{%
    \errmessage{(Inkscape) Transparency is used (non-zero) for the text in Inkscape, but the package 'transparent.sty' is not loaded}%
    \renewcommand\transparent[1]{}%
  }%
  \providecommand\rotatebox[2]{#2}%
  \ifx\svgwidth\undefined%
    \setlength{\unitlength}{49.54732089bp}%
    \ifx\svgscale\undefined%
      \relax%
    \else%
      \setlength{\unitlength}{\unitlength * \real{\svgscale}}%
    \fi%
  \else%
    \setlength{\unitlength}{\svgwidth}%
  \fi%
  \global\let\svgwidth\undefined%
  \global\let\svgscale\undefined%
  \makeatother%
  \begin{picture}(1,1.00059637)%
    \put(0,0){\includegraphics[width=\unitlength,page=1]{PushOff.pdf}}%
  \end{picture}%
\endgroup%